\documentclass[12pt]{amsart}

\usepackage{amssymb, hyperref}

\usepackage[all]{xy}

\theoremstyle{plain}

\newtheorem{theorem}{Theorem}
\newtheorem{lemma}[theorem]{Lemma}
\newtheorem{proposition}[theorem]{Proposition}

\theoremstyle{definition}

\newtheorem{remark}[theorem]{Remark}

\newcommand\bG{{\mathbb G}}

\newcommand\bP{{\mathbb P}}
\newcommand\bQ{{\mathbb Q}}

\newcommand\bZ{{\mathbb Z}}

\newcommand\cD{{\mathcal D}}

\newcommand\gp{{\rm gp}}

\newcommand\id{{\rm id}}

\newcommand\spl{{\rm split}}

\newcommand\End{{\rm End}}

\newcommand\Gal{{\rm Gal}}
\newcommand\GL{{\rm GL}}

\newcommand\Hom{{\rm Hom}}

\newcommand\Spec{{\rm Spec}}
\newcommand\Stab{{\rm Stab}}

\numberwithin{equation}{section}

\title{Groupes alg\'ebriques tr\`es sp\'eciaux/Very special algebraic groups}

\author{Michel Brion and Emmanuel Peyre}

\date{}

\begin{document}

\bigskip \noindent
{\sc Rubrique~:} G\'eom\'etrie alg\'ebrique.

\bigskip \bigskip

\maketitle

\bigskip \noindent
{\sc R\'esum\'e.} Nous disons qu'un groupe alg\'ebrique lisse $G$ sur un corps 
$k$ est tr\`es sp\'ecial si pour toute extension de corps $K/k$, toute $K$-vari\'et\'e
homog\`ene sous $G_K$ a un point $K$-rationnel. On sait que tout groupe
lin\'eaire r\'esoluble scind\'e est tr\`es sp\'ecial. Dans cette note, nous obtenons
la r\'eciproque et nous discutons ses relations avec la classification birationnelle
des actions de groupes alg\'ebriques.

\bigskip \noindent
{\sc Abstract.} We say that a smooth algebraic group $G$ over a field $k$ 
is very special if for any field extension $K/k$, every $G_K$-homogeneous
$K$-variety has a $K$-rational point. It is known that every  
split solvable linear algebraic group is very special.
In this note, we show that the converse holds, and discuss its
relationship with the birational classification of algebraic group actions.

\section{Introduction}
\label{sec:int}

Consider a smooth algebraic group $G$ over a field $k$,
and a $G$-variety $X$. By a theorem of Rosenlicht, there exist 
a dense open $G$-stable subset $X_0 \subset X$ and a morphism
$f: X_0 \to Y$, such that the fiber of $f$ at any point $x \in X_0$
is the orbit of $x$; moreover, $f$ identifies the function field
of $Y$ with the field of $G$-invariant rational functions on $X$
(see \cite[Thm.~2]{Rosenlicht-I}, and \cite[Sec.~7]{BGR} 
for a modern proof). We say that the rational map 
$f : X \dasharrow Y$ is the rational quotient of $X$ by $G$.

From this, one readily derives a birational classification of 
$G$-varieties with prescribed invariant function field. To state it, 
we introduce some notation. 
Given a finitely generated field extension $K/k$, we consider pairs 
$(X,\iota)$, where $X$ is a $G$-variety and 
$\iota: K \stackrel{\simeq}{\longrightarrow} k(X)^G$ 
is an isomorphism of fields over $k$. We say that
two pairs $(X,\iota)$ and $(X',\iota')$ are equivalent, if
there exists a $G$-equivariant birational isomorphism
$\varphi : X \dasharrow X'$ such that the isomorphism
$\varphi^* : k(X')^G  \stackrel{\simeq}{\longrightarrow} k(X)^G$ 
satisfies $\varphi^* \circ \iota' = \iota$. We may now state:

\begin{proposition}\label{prop:bir}
There is a one-to-one correspondence between equivalence
classes of pairs $(X,\iota)$ as above, and isomorphism classes
of $G_K$-homogeneous $K$-varieties. 
\end{proposition}

This easy result (which is implicitly known, see e.g. \cite[\S 2.7]{PV})
motivates the consideration of those smooth algebraic groups
for which all rational quotients have rational sections. 
These are described as follows:

\begin{theorem}\label{thm:vs}
The following conditions are equivalent for a smooth
algebraic group $G$:

\begin{enumerate}

\item[{\rm (i)}] For any $G$-variety $X$, the rational quotient
$f: X \dasharrow Y$ has a rational section.

\item[{\rm (ii)}] For any field extension $K/k$, every
$G_K$-homogeneous $K$-variety has a $K$-rational point.

\item[{\rm (iii)}] $G$ has a composition series with quotients
isomorphic to $\bG_a$ or $\bG_m$.

\end{enumerate}

\end{theorem}

The equivalence (i)$\Leftrightarrow$(ii) follows readily from
Rosenlicht's theorem on rational quotients. The implication 
(iii)$\Rightarrow$(ii) is also due to Rosenlicht (see 
\cite[Thm.~10]{Rosenlicht-I}). The proof of the converse 
implication is the main contribution of this note.

The algebraic groups satisfying (iii) are exactly the split solvable
linear algebraic groups in the sense of \cite[Def.~6.33]{Milne}.
On the other hand, (ii) obviously implies that for any field
extension $K/k$, every $G$-torsor over $\Spec(K)$ is trivial.
By \cite{RT}, this is equivalent to $G$ being special as defined 
by Serre in \cite{Serre}, that is, every locally isotrivial 
$G$-torsor over a variety is Zariski locally trivial. For this reason, 
we will call \emph{very special} the algebraic groups satisfying (ii).

In fact, every split solvable linear algebraic group $G$ satisfies
a much stronger condition: for any field extension $K/k$,
every $G_K$-homogeneous variety is rational (as follows from
\cite[Thm.~5]{Rosenlicht-II}). Equivalently, the field extension
$k(X)/k(X)^G$ is purely transcendental for any $G$-variety $X$.
This yields a further characterization of very special groups. 

One may also consider algebraic groups $G$ that are possibly
non-smooth, and require that for any field extension $K/k$, 
every $G_K$-homogeneous $K$-scheme has a $K$-rational point
(where a scheme $X$ equipped with an action $a$ of $G$ is
said to be homogeneous if the morphism 
$\id \times a : G \times X \to X \times X$ is faithfully
flat). But the result is unchanged, since every $G$-torsor
over $\Spec(K)$ is trivial, and hence $G$ is smooth in view
of \cite[Prop.~2.3]{RT}.

This note is organized as follows. The proof of Proposition
\ref{prop:bir} is presented in Section \ref{sec:pp}. The 
implications (i)$\Leftrightarrow$(ii)$\Leftarrow$(iii) are proved
in Section \ref{sec:pt}, which also makes the first steps in
the proof of (ii)$\Rightarrow$(iii). In Section \ref{sec:vst},
we show that any very special torus is split. Together with
the fact that any special unipotent group is split
(see \cite[Thm.~1.1]{Nguyen}), this enables us to complete 
the proof of (ii)$\Rightarrow$(iii) in Section \ref{sec:compl}.

\medskip

\noindent
{\bf Notation and conventions.}
We fix a ground field $k$ and choose an algebraic closure $\bar{k}$. 
We denote by $k_s$ the separable closure of $k$ in $\bar{k}$,
and by $\Gamma_k$ the Galois group of $k_s/k$. Given a field
extension $K/k$ and a $k$-scheme $X$, we denote by $X_K$ the
$K$-scheme $X \times_{\Spec(k)} \Spec(K)$.

A \emph{variety} is an integral separated $k$-scheme of finite type.
An \emph{algebraic group} $G$ is a $k$-group scheme of finite type.
We say that $G$ is \emph{linear} if it is smooth and affine.

Given an algebraic group $G$, a $G$-\emph{variety} is a variety
$X$ equipped with a $G$-action,
\[ a : G \times X \longrightarrow X, 
\quad (g,x) \longmapsto g \cdot x. \]

We say that a $G$-variety $X$ is $G$-\emph{homogeneous} 
if $G$ is smooth, $X$ is geometrically reduced, 
and the morphism 
\[ \id \times a : G \times X \longrightarrow X \times X,
\quad (g,x) \longmapsto (x, g \cdot x) \]
is surjective. If in addition $X$ is equipped with a $k$-rational point 
$x$, then we say that $X$ is a $G$-\emph{homogeneous space};
then $X \simeq G/\Stab_G(x)$, where $\Stab_G(x)$ denotes the
stabilizer. 

Every homogeneous space is smooth and quasi-projective; thus, 
so is every homogeneous variety.

\section{Proof of Proposition \ref{prop:bir}}
\label{sec:pp}

Consider a pair $(X,\iota)$ and choose a dense open $G$-stable
subset $X_0 \subset X$ with quotient $f : X_0 \to Y$ as in 
Rosenlicht's theorem. Identifying $k(Y)$ with $K$ via $\iota$,
the generic fiber of $f$ is a $G_K$-homogeneous $K$-variety, say 
$Z_0$. If we replace $X_0$ with an open subset $X_1$ satisfying
the same properties, then $Z_0$ is replaced with another 
$G_K$-homogeneous $K$-variety $Z_1$, which is 
$G_K$-equivariantly birationally isomorphic to $Z_0$.
But every $G_K$-equivariant birational isomorphism
$Z_0 \dasharrow Z_1$ is an isomorphism: this is proved
in \cite[Lem.~4]{Demazure} for homogeneous spaces, and the
general case follows by Galois descent. So we obtain a
$G_K$-homogeneous $K$-variety $Z$; moreover, replacing
$(X,\iota)$ with an equivalent pair replaces $Z$ with
a $G_K$-equivariantly isomorphic variety.

Conversely, consider a $G_K$-homogeneous $K$-variety $Z$.
We may then choose an immersion of $Z$ in some projective 
space $\bP^n_K$. Also, choose a $k$-variety $Y$ with function field 
$K$ and consider the closure $W$ of $Z$ in $\bP^n_Y$. 
Then $W$ is a $k$-variety equipped
with a $k$-morphism $f: W \to Y$. The action map
$a : G_K \times_{\Spec(K)} Z \to Z$ is identified with 
a morphism $G \times_{\Spec(k)} Z \to Z$, which yields 
a rational action of $G$ on $W$ (since $W$ and $Z$ have
the same function field). By construction, the field of 
invariant rational functions on $W$ is identified with $K$.
We now use Weil's regularization theorem (see the main
result of \cite{Weil} and \cite[Thm.~1]{Rosenlicht-I}): $W$
is $G$-birationally isomorphic to a $G$-variety $X$.
This associates with $Z$ a pair $(X,\iota)$, unique up
to equivalence.

One may readily check that the two constructions above
are mutually inverse, by using again the fact that every
equivariant birational isomorphism between homogeneous
varieties is an isomorphism.

\section{Proof of Theorem \ref{thm:vs}: first steps}
\label{sec:pt}

We first show the equivalence (i)$\Leftrightarrow$(ii).
Under the correspondence described in the proof of 
Proposition \ref{prop:bir}, the rational sections
of $f: X \dasharrow Y$ correspond to the $K$-points
of the associated $G_K$-homogeneous variety. Thus, 
(i) is equivalent to the assertion that (ii) holds
for any finitely generated field extension of $k$.
Given an arbitrary field extension $K/k$ and a
$G_K$-homogeneous variety $Z$, there exist 
a finitely generated subextension $L/k$ and a
$G_L$-homogeneous variety $W$ such that $W_L \simeq Z$.
Then every $L$-rational point of $W$ yields 
a $K$-rational point of $Z$; this completes the proof.

To show the equivalence (ii)$\Leftrightarrow$(iii),
we begin with some easy observations. First, if $G$ is 
very special, then $G_K$ is a very special $K$-group 
for any field extension $K/k$. Further properties are gathered
in the following:

\begin{lemma}\label{lem:char}
Consider an exact sequence of algebraic groups
\[ 1 \longrightarrow N \longrightarrow G \longrightarrow Q 
\longrightarrow 1. 
\]

\begin{enumerate}

\item[{\rm (i)}] If $G$ is very special, then so is $Q$.

\item[{\rm (ii)}] If $N$ and $Q$ are very special, then so is $G$.
 
\end{enumerate}

\end{lemma}

\begin{proof}
(i) Just note that $Q$ is smooth and every $Q_K$-homogeneous
variety is homogeneous under the induced action of $G_K$.

(ii) Since $N$ and $Q$ are smooth, $G$ is smooth as well.
Let $K/k$ be a field extension, and $X$ a $G_K$-homogeneous
$K$-variety. Then there is a quotient $f: X \to Y = X/N_K$,
where $Y$ is a $Q_K$-homogeneous variety: indeed, if
$X$ has a $K$-rational point $x$, then $X \simeq G_K/H$
where $H := \Stab_{G_K}(x)$ and we may take for $f$ the
canonical morphism $G/H \to G/N_K \cdot H$. The case
of an arbitrary $G_K$-homogeneous variety $X$ follows from 
this by using Galois descent together with the 
smoothness and quasi-projectivity of $X$. 

Since $Q_K$ is very special, $Y$ has a $K$-rational point 
$y$. The fiber of $f$ at $y$ is a $K$-variety, homogeneous
under $N_K$. As the latter is very special, it follows
that this fiber has a $K$-rational point.
\end{proof}

Finally, note that a smooth commutative algebraic group
$G$ is very special if and only if for any field extension
$K/k$, every quotient of $G_K$ is special.

These observations yield a quick proof of the implication
(iii)$\Rightarrow$(ii): by Lemma \ref{lem:char} (ii), it suffices
to show that $\bG_a$ and $\bG_m$ are very special. Since 
these groups are commutative, it suffices in turn to show
that for any field extension $K/k$, every quotient of 
$\bG_{a,K}$ or $\bG_{m,K}$ is special. But every quotient
of $\bG_{a,K}$ is isomorphic to $\bG_{a,K}$ (see 
\cite[IV.2.1.1]{DG}), and likewise for $\bG_{m,K}$; moreover,
$\bG_a$ and $\bG_m$ are special. This yields the assertion.

One may show similarly that every $G_K$-homogeneous 
$K$-variety is rational, for any split solvable linear algebraic 
group $G$ and any field extension $K/k$.

We now start the proof of the implication (ii)$\Rightarrow$(iii)
with the following:

\begin{lemma}\label{lem:str}
Let $G$ be a very special algebraic group. Then $G$ is connected, 
linear and solvable.
\end{lemma}

\begin{proof}
As the assertions are invariant under field extensions, 
we may assume $k$ algebraically closed. Since $G$ is
special, it is connected and linear by \cite[Thm.~1]{Serre}.
Moreover, every quotient group of $G$ is special in view of
Lemma \ref{lem:char} (i). In particular, the largest
semisimple quotient $H$ of $G$ is special, as well as
the largest adjoint semisimple quotient $H/Z(H)$, where $Z(H)$
denotes the (scheme-theoretic) center. By a result of
Grothendieck (see \cite[Thm.~3]{Grothendieck}), 
the special semisimple groups are exactly the products of 
special linear groups and symplectic groups. In particular, 
every special adjoint semisimple group is trivial. Thus, 
so is $H/Z(H)$, and $G$ is solvable.
\end{proof}

\section{Very special tori}
\label{sec:vst}

Let $T$ be a torus. We denote by 
$M = \Hom_{k_s-\gp}(T_{k_s},\bG_{m,k_s})$ 
its character group; this is a free abelian group
of finite rank equipped with a continuous action
of the absolute Galois group $\Gamma = \Gamma_k$. 
By an unpublished result of Colliot-Th\'el\`ene (see 
\cite[Thm.~18]{Huruguen}; this result is implicitly
contained in \cite{CS}), $T$ is special if and only 
if the $\Gamma$-module $M$ is \emph{invertible}, i.e., 
a direct factor of a permutation $\Gamma$-module. From 
this, we derive a criterion for $T$ to be very special:

\begin{lemma}\label{lem:vst}
The following conditions are equivalent:

\begin{enumerate}

\item[{\rm (i)}] $T$ is very special.

\item[{\rm (ii)}] Every quotient group of $T$ is special.

\item[{\rm (iii)}] For any subgroup of finite index 
$\Gamma' \subset \Gamma$, every $\Gamma'$-submodule
$M' \subset M$ is invertible.

\end{enumerate}
 
\end{lemma}

\begin{proof}
(i)$\Rightarrow$(ii) This follows from Lemma \ref{lem:char}.

(ii)$\Rightarrow$(iii) The invariant subfield 
$K := k_s^{\Gamma'} \subset k_s$ is a finite
separable extension of $k$ with absolute Galois
group $\Gamma'$, and $T_K$ is 
the $K$-torus with character group $M$ viewed
as a $\Gamma'$-module. Moreover, the
$\Gamma'$-submodule $M'$ of $M$ corresponds
to a quotient torus $T'$ of $T_K$. By asumption,
$T'$ is special; thus, $M'$ is invertible.

(iii)$\Rightarrow$(i) Let $K/k$ be a field extension.
Then again, the character group of $T_K$ is $M$
equipped with its action of the absolute Galois group 
$\Gamma_K$. We claim that this action factors through
that of $\Gamma$.

To show this, note that the action of $\Gamma_K$
on $K_s$ stabilizes the subfields $k_s$ and $Kk_s$. 
This yields an exact sequence
\[ 1 \longrightarrow \Gal(K_s/K k_s) \longrightarrow
\Gamma_K = \Gal(K_s/K) \longrightarrow \Gal(Kk_s/K)
\longrightarrow 1.\]
Since $T_{k_s}$ is split, so is $T_{Kk_s}$ and hence
the action of $\Gamma_K$ on $M$ factors through
an action of $\Gal(Kk_s/K)$. The latter group may be identified
with a subgroup of $\Gal(k_s/k) = \Gamma$, proving 
the claim.

Every $T_K$-homogeneous $K$-variety $X$ is a torsor
under a quotient $T'$ of $T_K$, which in turn 
corresponds to a $\Gamma_K$-stable submodule of $M$.
By the claim and the assumption, it follows that
$T'$ is special, i.e., $X$ has a $K$-rational point.
\end{proof}

\begin{lemma}\label{lem:spl}
Let $T$ be a very special torus. Then $T$ is split. 
\end{lemma}

\begin{proof}
It suffices to show that $\Gamma$ acts trivially
on $M$. Equivalently, for any subgroup 
$\Gamma' \subset \Gamma$ acting on $M$ via a quotient 
of prime order $p \geq 2$, the $\Gamma'$-action 
on $M$ is trivial.

Denote by $C_p$ the cyclic group of order $p$. 
By Lemma \ref{lem:vst}, the $C_p$-module $M$
is invertible, as well as any submodule $M'$. 
In particular, there exist a $C_p$-module $N$ 
and two integers $a,b \geq 0$ such that
\[ M' \oplus N \simeq \bZ^a \oplus (\bZ C_p)^b \]
as $C_p$-modules. By localizing at the prime ideal
$(p) \subset \bZ$, we obtain an isomorphism of
$\bZ_{(p)} C_p$-modules
\[ M'_{(p)} \oplus N_{(p)} \simeq 
\bZ_{(p)}^a \oplus (\bZ_{(p)} C_p)^b. \]
Moreover, $\bZ_{(p)}$ and $\bZ_{(p)} C_p$ are
indecomposable $\bZ_{(p)} C_p$-modules (this is
obvious for $\bZ_{(p)}$; for $\bZ_{(p)} C_p$,
one uses the isomorphism of $\bQ C_p$-modules
\[ \bQ C_p \simeq \bQ \oplus V, \]
where $\bQ$ is a trivial module and $V$ is an irreducible
non-trivial module of dimension $p-1$; moreover, this
isomorphism is not defined over $\bZ_{(p)}$).
As the Krull-Schmidt theorem holds for 
$\bZ_{(p)} C_p$-modules (see \cite[Thm.~2]{Jones}),
there exist integers $c, d \geq 0$ such that
\[ M'_{(p)}  \simeq \bZ_{(p)}^c \oplus (\bZ_{(p)} C_p)^d. \]
In particular, if $M'$ is not fixed pointwise by $C_p$, 
then its rank (as a $\bZ$-module) is at least $p$.

Consider the $\bQ C_p$-module 
$M_{\bQ} := M \otimes_{\bZ} \bQ$. If $M$ is
not fixed pointwise by $C_p$, then $M_{\bQ}$
contains a $\bQ C_p$-module $W$ isomorphic to $V$
(since every simple $\bQ C_p$-module is isomorphic
to $\bQ$ or $V$). Thus, $M' := M \cap W$ is a 
$C_p$-submodule of $M$, not fixed pointwise
by $C_p$ and of rank $p-1$; a contradiction. 
So $\Gamma$ acts trivially on $M$ as desired.
\end{proof}

\begin{remark}\label{rem:ks}
The localization argument in the above proof cannot
be avoided, since the Krull-Schmidt theorem generally fails 
for $C_p$-modules. More specifically, we may choose $p$ 
so that the ring $R$ has nontrivial class group, and choose 
a non-principal ideal $A \subset R$. Then $A$ is a summand
of a free $R$-module, but is not isomorphic to $R$.
\end{remark}

\section{Completion of the proof of Theorem \ref{thm:vs}}
\label{sec:compl}

It remains to show the implication (ii)$\Rightarrow$(iii). 
Let $G$ be a very special algebraic group,
and recall from Lemma \ref{lem:str} that $G$ 
is connected, linear and solvable. Choose a 
maximal torus $T$ of $G$; then $T_K$ is a maximal
torus of $G_K$ for any field extension $K/k$
(see \cite[Thm.~17.82]{Milne}).

If $k$ is perfect, then $G = U \rtimes T$,
where $U$ denotes the unipotent radical of $G$ 
(indeed, $G_{\bar{k}} = U_{\bar{k}} \rtimes T_{\bar{k}}$,
see e.g. \cite[Thm.~16.33]{Milne}). 
As a consequence, $T$ is a quotient group of $G$, 
and hence is very special by Lemma \ref{lem:char} (i). 
In view of Lemma \ref{lem:spl}, it follows that $T$ is split. 
But the smooth connected unipotent group $U$ is split 
as well (see e.g. \cite[Cor.~16.23]{Milne}), 
and hence so is $G$.

For an arbitrary field $k$, consider the derived subgroup 
$\cD(G)$; this is a smooth connected unipotent normal subgroup 
of $G$ (see e.g. \cite[Cor.~6.19, Prop.~16.34]{Milne}).
The quotient group $G/\cD(G)$ lies in a unique exact sequence
of commutative algebraic groups
\begin{equation}\label{eqn:ext} 
0 \longrightarrow M \longrightarrow G/\cD(G) 
\longrightarrow V \longrightarrow 0, 
\end{equation}
where $M$ is of multiplicative type and 
$V$ is unipotent; moreover, (\ref{eqn:ext})
splits uniquely over the perfect closure $k_i$ of $k$
(see e.g. \cite[Thm.~16.3]{Milne}). As a consequence, 
the natural morphism $T \to G/\cD(G)$ induces an
isomorphism $T \stackrel{\simeq}{\longrightarrow} M$. 
Also, $V$ is a quotient group of $G$, and hence 
is special. Since $V$ is unipotent, 
it is split by \cite[Thm.~1.1]{Nguyen}.
In view of \cite[Lem.~5.7]{Conrad}, it follows that
the exact sequence (\ref{eqn:ext}) has a unique splitting. 
Thus, we may identify $G/\cD(G)$ with $T \times V$.
In particular, $T$ is a quotient group of $G$.
Using Lemmas \ref{lem:char} and \ref{lem:spl},
it follows that $T$ is split. 

Denote by $U$ the pull-back of $V$ in $G$.
Then $U$ is a smooth connected unipotent group,
and $G \simeq U \rtimes T$. By \cite[Thm.~3.7]{Conrad},
$U$ has a largest split subgroup $U_{\spl}$; 
moreover, the formation of $U_{\spl}$ is compatible
with separable field extensions. As a consequence, 
$U_{\spl}$ is normal in $G$. Also, $T$ acts trivially
on $U/U_{\spl}$ in view of \cite[Cor.~4.4]{Conrad}. 
Thus, $G/U_{\spl} \simeq U/U_{\spl} \times T$.
In particular, $U/U_{\spl}$ is a quotient group of $G$, 
and hence is special. Using again 
\cite[Thm.~1.1]{Nguyen}, it follows that
$U$ is split, and hence so is $G$.

\begin{remark}\label{rem:final}
By inspecting the proof of Theorem \ref{thm:vs},
one may check that the conditions (i), (ii), (iii) 
are equivalent to

\begin{enumerate}
\item[{\rm (iv)}] 
{\it Every quotient group of $G$ is special}.
\end{enumerate}

If $G$ is linear, they are also equivalent to

\begin{enumerate}
\item[{\rm (i)'}] 
{\it The rational quotient map $V \dasharrow V/G$
has a rational section for any finite-dimensional
representation $G \to \GL(V)$}.
\end{enumerate}

Indeed, the implication (i)$\Rightarrow$(i') is obvious. 
We show that (i)'$\Rightarrow$(iv): let $Q$ be a quotient 
group of $G$, and choose a faithful finite-dimensional 
representation $\rho : Q \to \GL(W)$.
Then $G$ acts on $\GL(W)$ by right multiplication via $\rho$.
Moreover, $\GL(W)$ may be viewed as an open subset 
of $V := \End(W)$, stable by the linear representation of $G$ 
in $V$ by right multiplication via $\rho$ again. 
Thus, the rational quotient map $V \dasharrow V/G$
may be viewed as the $Q$-torsor $\GL(W) \to \GL(W)/Q$. 
By assumption, this torsor has a rational section,
and hence is locally trivial for the Zariski topology.
It follows that $Q$ is special by using 
\cite[Thm.~2]{Serre} (which is obtained over an
algebraically closed field, but whose proof holds
unchanged over an arbitrary field).
\end{remark}

\medskip \noindent
{\bf Acknowledgements.}
We thank Jean-Louis Colliot-Th\'el\`ene, Zinovy Reichstein
and an anonymous referee for helpful comments and suggestions. 
Our original proof of Lemma \ref{lem:vst} used the classification 
of indecomposable $C_p$-modules (see \cite[\S 74]{CR}). 
Colliot-Th\'el\`ene sent us an alternative proof, which avoids 
this classification. The proof presented here owes much to 
a suggestion of the referee. 

\bibliographystyle{amsalpha}

\bigskip \bigskip

\bigskip \bigskip

\noindent
\address{Michel BRION~:
Universit\'e Grenoble Alpes, Institut Fourier, CS 40700,
38058 Grenoble Cedex 09.\\ 
T\'el. 04 76 51 42 98, Fax 04 76 51 44 78.\\
Courriel : Michel.Brion@univ-grenoble-alpes.fr}

\bigskip \bigskip

\noindent
\address{Emmanuel PEYRE~: 
Universit\'e Grenoble Alpes, Institut Fourier, CS 40700,
38058 Grenoble Cedex 09.\\ 
T\'el. 04 76 51 44 39, Fax 04 76 51 44 78.\\
Courriel : Emmanuel.Peyre@univ-grenoble-alpes.fr}

\end{document}